\documentclass[11pt]{amsart}
\usepackage{amsfonts}

\newcommand{\<}{\langle}
\renewcommand{\>}{\rangle}
\newcommand{\R}{{\mathbb R}}
\newcommand{\N}{{\mathbb N}}
\newcommand{\eps}{{\varepsilon}}
\renewcommand{\phi}{{\varphi}}
\newcommand{\rcirc}{{\operatorname{\overset{\hphantom{i}_\circ}{\textit R}}}}
\newcommand{\Ric}{{\operatorname{{Ric}}}}
\newcommand{\scal}{{\operatorname{{scal}}}}
\newcommand{\Id}{{\operatorname{Id}}}
\newcommand{\Tr}{{\operatorname{Tr}}}

\newcommand{\dvol}{{\operatorname{\textit{dvol}}}}
\newcommand{\vol}{{\operatorname{vol}}}

\newcommand{\restr}[1]{\lower0.4ex\hbox{$|$}\lower0.7ex
  \hbox{$\scriptstyle{#1}$}}

\newtheorem{theorem}{Theorem}
\newtheorem{lemma}[theorem]{Lemma}
\newtheorem{proposition}[theorem]{Proposition}
\newtheorem{corollary}[theorem]{Corollary}
\newtheorem{maintheorem}[theorem]{Main Theorem}

\theoremstyle{definition}

\newtheorem{remark}[theorem]{Remark}

\newcommand{\Section}[1]{\section{#1}\setcounter{theorem}{0}}

\begin{document}

\title[Local symmetry of harmonic spaces as determined by spectra of geodesic spheres]
{Local symmetry of harmonic spaces as determined by the spectra of small geodesic
spheres}

\author{Teresa Arias-Marco}
\address{Departamento de Matem\'aticas, Universidad de Extremadura,
06071 Badajoz, Spain}
\email{ariasmarco@unex.es}

\author{Dorothee Schueth}
\address{Institut f\"ur Mathematik, Humboldt-Universit\"at zu
Berlin, D-10099 Berlin, Germany}
\email{schueth@math.hu-berlin.de}


\thanks{The authors were partially supported by DFG Sonderforschungsbereich~647.
The first author's work has also been supported by
D.G.I.~(Spain) and FEDER Project MTM2010-15444, by Junta de Extremadura and FEDER
funds, and the program
``Estancias de movilidad en el extranjero `Jos\'e Castillejo' para j\'ovenes doctores''
of the Ministry of Education (Spain).\\
\hbox to\the\parindent{}\textit{Dedication.} Dorothee Schueth would like to
dedicate this article to her former high school teacher, Martin Berg.
It was he who first made her see the beauty of Mathematics.}

\begin{abstract}
We show that in any harmonic space, the eigenvalue spectra of the
Laplace operator on small geodesic spheres
around a given point determine the norm $|\nabla R|$
of the covariant derivative
of the Riemannian curvature tensor in that point. In particular,
the spectra of small geodesic spheres in a harmonic space
determine whether the space is locally symmetric.
For the proof we use the first few
heat invariants and consider certain coefficients
in the radial power series expansions of the curvature
invariants $|R|^2$ and $|\Ric|^2$ of the geodesic
spheres.
Moreover, we obtain analogous results for geodesic balls with
either Dirichlet or Neumann boundary conditions.
We also comment on the relevance of these results to constructions
of Z.I.~Szab\'o.
\end{abstract}

\maketitle

\Section{Introduction}
\label{sec:intro}

\noindent
For a compact closed Riemannian manifold $S$ the {\it spectrum\/} of~$S$
is the eigenvalue spectrum, including multiplicities,
of the associated (positive semi-definite)
Laplace operator~$\Delta$ acting on smooth functions.
A central question of inverse spectral geometry asks to which extent
the geometry of~$S$ is determined by its spectrum. The so-called
{\it heat invariants\/} $a_k(S)$ of~$S$ are examples of geometric
invariants which are determined by the spectrum of~$S$; indeed,
they are the coefficients in the famous asymptotic expansion
by Minakshisundaram-Pleijel,
$$
\Tr(\exp(-t\Delta))\sim(4\pi t)^{-\dim(S)/2}\sum_{k=0}^\infty a_k(S) t^k
$$
for $t\downarrow0$. The first few of these coefficients are given
by
\begin{equation*}
a_0(S)=\vol(S),\mbox{\ \ }
a_1(S)=\frac16\int_S\scal\,\dvol_S\,,\mbox{\ \ }
a_2(S)=\frac1{360}\int_S(5\scal^2-2|\Ric|^2+2|R|^2)\dvol_S\,,
\end{equation*}
where $\scal$, $\Ric$, and $R$ denote the scalar curvature, the Ricci operator,
and the Riemannian curvature operator of~$S$, respectively. In general,
each $a_k(S)$ is the integral over~$S$ of certain curvature invariants; see~\cite{Gi}
for more information.

Nevertheless, there exist many examples of pairs or families
of {\it isospectral\/} Riemannian manifolds
(i.e., sharing the same spectrum) which are not isometric, sometimes not
even locally isometric;
see, for example, the survey article~\cite{Go00}. Still, many questions remain open;
for example, it is not known whether a locally symmetric compact closed Riemannian
manifold can be isospectral to a locally nonsymmetric Riemannian manifold.

On the other hand, the geometry of {\it geodesic spheres\/} plays an interesting
role in Riemannian geometry. Chen and Vanhecke~\cite{CV}
formulated the following general question: To what extent do the properties of small
geodesic spheres determine the Riemannian geometry of the ambient space?
For example, Gray and Vanhecke~\cite{GV} studied the information contained
in the volume function of small geodesic spheres and investigated the
question whether a Riemannian manifold whose geodesic spheres have the
same volumes as spheres in euclidean space must necessarily be flat
(answering this question in the positive under various choices of
additional assumptions).

In the context of inverse spectral geometry,
an interesting special version of the above question is: To what extent do the spectra of
small geodesic spheres in a (possibly noncompact) Riemannian manifold~$M$
determine the geometry of~$M$\,?
For example, Theorem~6.18 in~\cite{CV} uses the information contained
in the heat invariants $a_0$ and~$a_1$ of small geodesic spheres (viewed as
functions of the radius) and concludes local isometry of manifolds with adapted holonomy
to certain model spaces under the assumption that all small geodesic spheres
around each point are isospectral to the corresponding geodesic spheres in those
model spaces.

In order to arrive at such and similar results, one uses
radial power series expansions of curvature invariants, both of the ambient
space and of the geodesic spheres.
In general, even the first few coefficients of such expansions become very
complicated; see, for example, the various formulas in \cite{GV} or~\cite{CV}.
One setting in which a quite restrictive geometric assumption on the ambient
space makes the calculations considerably easier is the setting of {\it harmonic\/}
ambient spaces.

A manifold is called harmonic if the volume density function
of the geodesic exponential map is radial around each point. The notion of harmonicity
was first introduced by Copson and Ruse~\cite{CR} and 
intensively studied by Lichnerowicz~\cite{Li44}; see also~\cite{RWW}.
Chapter~6 of the book by
Besse~\cite{Be} gives a useful survey of properties of harmonic spaces.
One of the important facts about harmonic spaces is that
they are Einstein~\cite{Be} and hence analytic~\cite{DK}
(the latter result was not yet known when~\cite{Be} was written).
A locally symmetric manifold is harmonic if and only if it is flat or of rank one;
the famous Lichnerowicz conjecture postulated that, conversely,
each harmonic space
is locally symmetric; i.e., satisfies $\nabla R=0$ (this condition is
classically known to be equivalent to the condition that the local
geodesic symmetries around each point be isometries).
For the case of compact manifolds with finite fundamental group
the Lichnerowicz conjecture was proved by Szab\'o~\cite{Sz90};
however, Damek and Ricci gave examples of noncompact homogeneous
harmonic manifolds which are not locally symmetric
in infinitely many dimensions greater or equal to seven~\cite{DR}.
These spaces are usually referred to as {\it Damek-Ricci spaces};
see~\cite{BTV} for more information.

Specializing the above question about the information contained
in the spectra of small geodesic spheres to the
setting of harmonic spaces,
we are able to prove in the present paper that the spectra of 
small geodesic spheres in a harmonic space determine whether the
space is locally symmetric (see Corollary~\ref{cor:main} below).
More precisely, we obtain:

\begin{maintheorem}
\label{thm:main}
Let $M_1$ and $M_2$ be harmonic spaces, and let $p_1\in M_1$,
$p_2\in M_2$. If there exists $\eps>0$ such that for each
$r\in(0,\eps)$ the geodesic spheres $S_r(p_1)$ and $S_r(p_2)$
are isospectral, then $|\nabla R|_{p_1}^2=|\nabla R|_{p_2}^2$.
\end{maintheorem}

\begin{corollary}
\label{cor:main}
Let $M_1$ and $M_2$ be harmonic spaces. Assume that the
hypothesis of Theorem~\ref{thm:main} is satisfied for
\emph{each} pair of points $p_1\in M_1$, $p_2\in M_2$.
Then $M_1$ is locally symmetric if and only if $M_2$ is locally
symmetric.
\end{corollary}

In particular, note that in the case of {\it locally homogeneous\/} harmonic spaces
$M_1$ and~$M_2$, the hypothesis of Theorem~\ref{thm:main}
implies that $M_1$ is locally symmetric if and only if $M_2$ is locally symmetric.
Actually, all known examples of harmonic spaces are locally homogeneous;
it is an open question whether there exist harmonic spaces which are not
locally homogeneous.

Interestingly, our result implies that certain pairs of geodesic
spheres which were claimed to be isospectral by Szab\'o in \cite{Sz01},~\cite{Sz05}
are actually {\it not\/} isospectral. In fact, Szab\'o considered
(as the featured examples in a more general construction) geodesic
spheres in certain symmetric spaces~$M_1$ (namely, quaternionic hyperbolic
space of real dimension $4m\ge12$) and in certain associated
locally nonsymmetric Damek-Ricci
spaces $M_2$ of the same dimension (see also Remark~\ref{rem:damekricci}).
He stated that every pair
of geodesic spheres $S_r(p_1)\subset M_1$ and $S_r(p_2)\subset M_2$
of the same radius was isospectral. Since these
ambient manifolds $M_1$ and $M_2$ are harmonic and homogeneous, and $M_1$ is
locally symmetric while $M_2$ is not, Corollary~\ref{cor:main} immediately
implies that Szab\'o's result cannot be correct. Note that it was
F\"urstenau~\cite{Fu} who first
discovered that actually there was a gap in Szab\'o's isospectrality argument.
The question of whether that proof could be repaired or not had since
remained open; our result settles this question in the negative.

The incorrect examples of geodesic spheres mentioned above had the
notable property that one is homogeneous and the other not.
While it remains unknown whether a homogeneous metric on a sphere
can be isospectral to a non-homogeneous one, Szab\'o in an earlier article~\cite{Sz99}
did construct a pair of isospectral metrics, only one of which is homogeneous,
on the product of a sphere and a torus (those results are not
affected by the error in the later papers).

We obtain analogs of our above results for geodesic balls endowed with
either Dirichlet or Neumann boundary conditions; see Theorem~\ref{thm:balls}
and Corollary~\ref{cor:balls}. Similarly as above, this implies
that Szab\'o's examples in \cite{Sz01},~\cite{Sz05}
of isospectral geodesic balls (of any given radius)
in quaternionic hyperbolic space of real dimension at least~$12$ and in certain
associated locally nonsymmetric Damek-Ricci spaces were erroneous.

Note that nevertheless there do exist isospectral pairs and even
continuous families of isospectral metrics on spheres and balls;
the first such examples were due to Gordon~\cite{Go01}.

In order to prove Theorem~\ref{thm:main} we use the heat invariants $a_0(S_r(p))$ and
$a_2(S_r(p))$ of geodesic spheres in harmonic spaces.
In particular, we study the coefficients
of~$r^2$ in the radial power series expansions of
$\frac1{\vol(S_r(p))}\int_{S_r(p)}|\Ric^S|^2\dvol_{S_r(p)}$ and $\frac1{\vol(S_r(p))}\int_{S_r(p)}|R^S|^2\dvol_{S_r(p)}$,
where $\Ric^S$ and $R^S$ denote the Ricci operator and the Riemannian curvature
operator of~$S_r(p)$. From the form of these coefficients (see
Proposition~\ref{prop:coeffs} and its mean value version
Proposition~\ref{prop:intcoeffs}), we are
able to conlude that the heat invariants $a_0$ and $a_2$ of $S_r(p)$, viewed
as functions of~$r$, together
determine the value of $|\nabla R|^2$ at the midpoint~$p$.
Note that the same is not true for $a_0(S_r(p))$ alone;
see Remark~\ref{rem:damekricci}. Moreover,
in the harmonic setting,
the function $r\mapsto a_1(S_r(p))=\frac16\int_{S_r(p)}\scal^S\,\dvol_{S_r(p)}$
does actually not contain more information
than $r\mapsto a_0(S_r(p))=\vol(S_r(p))$; see Remark~\ref{rem:volscal}.
So it is indeed necessary for our purpose to consider $a_2(S_r(p))$.

Our computations rely heavily on the harmonicity of the ambient space.
Note that they are related to certain more general computations in~\cite{GV} and~\cite{CV};
for example, Theorem~8.1 of~\cite{CV} actually includes a kind of analog to
our Proposition~\ref{prop:coeffs}, and this even for general, not only for
harmonic manifolds; however, that theorem
contains information only on the coefficients of~$r^j$ with $j\le0$,
while we need the coefficients of~$r^2$.
In fact, in the harmonic case, the lower order coefficients turn out to be
determined already by the function $r\mapsto a_0(S_r(p))=\vol(S_r(p))$; see
Proposition~\ref{prop:coeffs} and Remark~\ref{rem:volscal}(i).

This paper is organized as follows:

In Section~\ref{sec:prelim} we gather the necessary background on harmonic
spaces, mostly following~\cite{Be}; in particular, we recall
Ledger's recursion formula for the power series expansion of the
second fundamental form of geodesic spheres, and the resulting
curvature identities in harmonic spaces.

In Section~\ref{sec:curv}, we study the coefficients of $r^j$ for $j\le2$
in $|\Ric^{S_r(p)}|^2_{\exp(ru)}$ and $|R^{S_r(p)}|^2_{\exp(ru)}$ for unit
tangent vectors~$u$ of harmonic spaces, using the power series expansion
of the second fundamental form and its radial covariant derivative,
as well as the Taylor series expansion of the Riemannian curvature tensor.
Proposition~\ref{prop:coeffs} is the main result of this section.

Section~\ref{sec:proof} is devoted to the proof of the Main Theorem~\ref{thm:main}.
In preparation for this, we first derive a mean value version of
Proposition~\ref{prop:coeffs}; see Proposition~\ref{prop:intcoeffs}.

Finally, in Section~\ref{sec:balls}, we prove the analog of Theorem~\ref{thm:main}
for geodesic balls. We consider the heat coefficients
of geodesic balls in harmonic spaces
and show that the functions $r\mapsto a_0(B_r(p))$ and $r\mapsto a_2(B_r(p))$
(either for Dirichlet or for Neumann boundary conditions) together
determine the value of $|\nabla R|^2$ at the midpoint~$p$ of the balls.
More precisely, we show that the coefficient of~$r^3$ in the radial
power series expansion of the quotient $a_2(B_r(p))/a_0(B_r(p))$ is a sum
of a nonzero multiple of~$|\nabla R|^2_p$ and
of terms determined by the function $r\mapsto a_0(B_r(p))$.

\Section{Preliminaries}
\label{sec:prelim}

\subsection{Volume density and the shape operator of geodesic spheres}\

\noindent
In the following, let $M$ be a complete, connected, $n$-dimensional Riemannian manifold.
For $p\in M$, let $\exp_p=\exp\restr{T_pM}:T_pM\to M$ denote the associated geodesic
exponential map. For a vector $v\in T_pM$ we denote by $\gamma_v$ the geodesic
with initial velocity~$v$. Identifying $T_v(T_pM)$ with $T_pM$, we regard
the differential $d(\exp_p)_v$ as a linear map from $T_pM$ to $T_{\exp v}M$.
We denote parallel translation along~$\gamma_v$ by
$P_{\gamma_v}^{s,t}:T_{\gamma_v(s)}M\to T_{\gamma_v(t)}M$.
Given any unit vector $u\in S_1(0_p):=\{u\in T_pM\mid |u|=1\}$
and $r\in\R$, we consider the volume density
$$\theta_u(r):=
\det\bigl(P_{\gamma_u}^{r,0}\circ d(\exp_p)_{ru}\bigr).
$$
Note that $\theta_u(r)$ is the infinitesimal volume distortion
of the map $\exp_p$ at the point $ru\in T_pM$. Recall the Gauss lemma:
The vector $d(\exp_p)_{ru} u$ is a unit vector perpendicular
to each $d(\exp_p)_{ru} w$ with $w\perp u$. Thus, for each $r\in(0,i(p))$,
where $i(p)$ denotes the injectivity radius of~$M$ at~$p$,
\begin{equation}
\label{eq:vtheta}
v_u(r):=r^{n-1}\theta_u(r)
\end{equation}
is the infinitesimal volume distortion at $u$ of the map
$$S_1(0_p)\ni u\mapsto\gamma_u(r)=\exp(ru)\in S_r(p),
$$
where $S_r(p)\subset M$ denotes the geodesic sphere of radius~$r$ around~$p$.
Let $\sigma_u(r)$ denote the shape operator of $S_r(p)$ at~$\exp(ru)$;
that is,
$$\sigma_u(r):=(\nabla\nu)\restr{T_{\exp(ru)}M}\,,
$$
where $\nabla$ is the Levi-Civita connection of~$M$ and
$\nu$ denotes the outward pointing unit normal vector field on
the geodesic ball $B_{i(p)}(p)\setminus\{p\}$.
In particular, $\nu\circ\gamma_u=\dot\gamma_u$, $\sigma_u\nu=0$,
and the image of $\sigma_u(r)$ is contained in $T_{\gamma_u(r)}S_r(p)$.
It is well-known that for all $r\in(0,i(p))$,
\begin{equation}
\label{eq:volgrowth}
v'_u(r)/v_u(r)=\Tr(\sigma_u(r)),
\end{equation}
and that the covariant derivative~$\sigma'_u$ of the
endomorphism field $\sigma_u$ along $\gamma_u\restr{(0,i(p))}$
satisfies the so-called Riccati equation
\begin{equation}
\label{eq:riccati}
\sigma'_u=-\sigma_u^2-R_{\dot\gamma_u},
\end{equation}
where $R$ is the Riemannian curvature tensor of $M$, given by
$R(x,y)z=-\nabla_x\nabla_yz+\nabla_y\nabla_xz+\nabla_{[x,y]}z$, and where
$R_\nu:=R(\nu,\,.\,)\nu$. (Note that here we use the same sign for~$R$
as Besse~\cite{Be}.) Let $C_u(r):=r\sigma_u(r)$. This endomorphism
field along $\gamma_u\restr{(0,i(p))}$
is smoothly extendable to $r=0$ by $C_u(0):=I_u$, where $I_u$ is defined
by $I_u(u)=0$ and $I_u\restr{\{u\}^\perp}=\Id_{\{u\}^\perp}$.
Moreover, from~(\ref{eq:riccati})
one can derive Ledger's recursion formula for the covariant derivatives
of $C_u$ at $r=0$ (see, e.g.,~\cite{CV}):
$$
(k-1)C_u^{(k)}(0)=-k(k-1)R_u^{(k-2)}-\sum_{\ell=0}^k\binom k\ell C_u^{(\ell)}(0)C_u^{(k-\ell)}(0)
$$
for all $k\in\N$, where $R_u^{(k)}$ is
the $k$-th covariant derivative of the
endomorphism field $R_{\dot\gamma_u}$ along~$\gamma_u$ at $r=0$. This formula allows
one to successively compute the $C_u^{(k)}(0)$ in terms of the
en\-do\-mor\-phisms~$R_u^{(k)}$ of $T_pM$. Forming
the Taylor series of $C_u$ and dividing by~$r$, one obtains (see, e.g., \cite{Be}, \cite{CV}):
\begin{equation}
\label{eq:pow}
\begin{split}
P^{r,0}_{\gamma_u}\circ\sigma_u(r)\circ P^{0,r}_{\gamma_u}={}&\frac1r I_u-\frac r3 R_u
-\frac{r^2}4 R'_u -\bigl(\frac1{10}R''_u+\frac1{45}R_uR_u\bigr)r^3\\
&-\bigl(\frac1{36}R'''_u+\frac1{72}R_uR'_u
+\frac1{72}R'_uR_u\bigr)r^4\\
&-\bigl(\frac1{168}R^{(4)}_u+\frac1{210}R_uR''_u+\frac1{210}R''_uR_u+\frac1{112}R'_uR'_u
+\frac2{945}R_uR_uR_u\bigr)r^5\\
&+O(r^6).
\end{split}
\end{equation}

\subsection{Curvature identities in harmonic spaces}\

\noindent
The manifold $M$ is called a \emph{harmonic space} if for every $p\in M$
the above function $\theta_u$ does not depend on $u\in S_1(0_p)$.
An equivalent condition is that for all $r\in(0,i(p))$, the geodesic spheres $S_r(p)$ have
constant mean curvature (recall equations~(\ref{eq:vtheta}), (\ref{eq:volgrowth})).
For more information on harmonic spaces see \cite{RWW} or~\cite{Be}.
If $M$ is harmonic then the function $\theta_u$ does in fact not even depend on~$p$; that is,
there exists $\theta:[0,\infty)\to\R$ such that
$$\theta_u(r)=\theta(r)
$$
for all $u\in TM$
with $|u|=1$.
Moreover, even the local or infinitesimal versions of the above
condition imply that the manifold is Einstein~\cite{Be} and therefore
analytic~\cite{DK}. Hence, the local or infinitesimal versions of the above conditions
are equivalent to the global versions.
Since $\theta_u(r)$ depends only on~$r$, so does $v_u(r)$ and
hence $\Tr(\sigma_u(r))$.
From this one can successively derive, using the expansion~(\ref{eq:pow}):

\begin{proposition} [see \cite{Be}, Chapter~6]
\label{prop:constants}
If $M$ is harmonic then there exist constants $C,H,L\in\R$ such that for all $p\in M$ and
all $u\in T_pM$ with $|u|=1${\rm:}
\begin{itemize}
\item[(i)] $\Tr(R_u)=C$; in particular{\rm:}
\item[(ii)] $\Tr(R^{(k)}_u)=0$ for all $k\in\N$.
\item[(iii)] $\Tr(R_uR_u)=H$; in particular{\rm:}
\item[(iv)] $\Tr(R_uR'_u)=0$ and
\item[(v)] $\Tr(R_uR''_u)=-\Tr(R'_uR'_u)$.
\item[(vi)] $\Tr(32R_uR_uR_u-9R'_uR'_u)=L$.
\end{itemize}
\end{proposition}

In fact, taking traces in~(\ref{eq:pow}),
one has in the harmonic case:
\begin{equation}
\label{eq:powtrharm}
\Tr(\sigma_u(r))=(n-1)\frac1r-\frac 13 Cr
-\frac1{45}H r^3-\frac1{15120}Lr^5+O(r^7)
\end{equation}
for $r\downarrow0$ and
all $u\in TM$ with $|u|=1$.
Note that Proposition~\ref{prop:constants}(i) just says that the
Einstein constant of~$M$ is~$C$; that is, $\Ric=C\Id$ on each $T_pM$.
Recall that the Ricci operator
is defined by $\<\Ric(x),y\>=\Tr(R(x,\,.\,)y)$ for all $x,y\in T_pM$ and all $p\in M$.
From Proposition~\ref{prop:constants} one can further derive:

\begin{proposition} [see~\cite{Be}, Chapter 6]
\label{prop:reqs}
If $M$ is harmonic, then for the above constants $C,H,L$ and each
$p\in M${\rm:}
\begin{itemize}
\item[(i)]
$\<R(x,\,.\,)\,.\,,R(y,\,.\,)\,.\,\>=\frac23((n+2)H-C^2)\<x,y\>$ for all $x,y\in T_pM$;
in particular{\rm:}
\item[(ii)]
$|R|_p^2=\frac23n((n+2)H-C^2)$.
\item[(iii)]
$32\bigl(nC^3+\frac92C|R|_p^2+\frac72\hat R(p)-\rcirc(p)\bigr)-27|\nabla R|_p^2=n(n+2)(n+4)L$.
\end{itemize}
\end{proposition}

Here, the functions $\hat R, \rcirc\in C^\infty(M)$ are certain
curvature invariants of order six which are defined as follows:
If $\{e_1,\ldots,e_n\}$ is an orthonormal basis of $T_pM$ and
$R_{ijk\ell}:=\<R(e_i,e_j)e_k,e_\ell\>$, then
$$\rcirc(p):=\sum_{i,j,k,\ell,a,b} R_{ijk\ell}R_{ja\ell b}R_{aibk},\;\;\;\;
\hat R(p):=\sum_{i,j,k,\ell,a,b} R_{ijk\ell}R_{k\ell ab}R_{abij}.
$$
Note that the term $nC^3$ in Proposition~\ref{prop:reqs}(iii) reads
$nC^2$ in the corresponding equation~6.67 in~\cite{Be}, but this was
obviously a misprint (note that curvature terms of different order cannot occur
here);
see also formula (3.1) in~\cite{Wa}.

Proposition~\ref{prop:reqs}(iii) will be used in Section~\ref{sec:proof},
together with the following formula which actually holds in
any Einstein manifold; see formula~(6-7) in \cite{Li58} or formula~(11.3) in~\cite{GV}:
\begin{equation}
\label{eq:lichn}
-\frac12\Delta(|R|^2)=2C|R|^2-\hat R-4\rcirc+|\nabla R|^2,
\end{equation}
where $\Delta$ denotes the Laplace operator on functions,
that is, $\Delta f=-\sum_i \bigl(e_i(e_if)-(\nabla_{e_i}e_i)f\bigr)$ for local orthonormal
frames~$\{e_1,\ldots,e_n\}$. (Again, there is a misprint in two of the coefficients
in the corresponding formula~6.65 in~\cite{Be}.) If $M$ is harmonic, then
the left hand side of~(\ref{eq:lichn}) is zero by Proposition~\ref{prop:reqs}(ii).
Finally, we recall the following well-known observations
which will be used in Section~\ref{sec:proof}:

\begin{remark}
\label{rem:volscal}
Let $M$ be an $n$-dimensional harmonic space with volume density function $\theta$ as above.

(i)
For any $p\in M$, the volume of the geodesic sphere~$S_r(p)$ with $0<r<i(p)$
equals the volume~$\omega_{n-1}$ of the standard unit sphere $S^{n-1}$
in~$\R^n$ multiplied by the factor
$$v(r):=r^{n-1}\theta(r)
$$
Note that
$v(r)=v_u(r)$
for each unit vector $u\in TM$, where $v_u$ is the function defined in~(\ref{eq:vtheta}).
The function~$v$
determines the volume growth
function $v'/v$ of the geodesic spheres,
and thus it determines, by~(\ref{eq:volgrowth}), the function
$\Tr(\sigma_u(r))$ (which is independent of~$u$). By~(\ref{eq:powtrharm}),
the function which associates to small values of~$r$ the volume of geodesic spheres of radius~$r$
in a given harmonic space~$M$ determines the constants
$C,H,L$ (and of course $n$) associated with~$M$.

(ii)
Let $\scal=nC$ denote the scalar curvature of~$M$. Let $p\in M$, fix
some $r\in(0,i(p))$, and let $\scal^S$ denote the scalar curvature function of~$S_r(p)$.
A routine calculation using the Gauss equation shows that for each unit
vector $u\in T_pM$ we have
$$
\scal^S(\exp(ru))= \scal -2\<\Ric(\dot\gamma_u(r)),\dot\gamma_u(r)\>+(\Tr(\sigma_u(r)))^2
  -\Tr(\sigma_u(r)^2)
$$
which by the Einstein condition and equations (\ref{eq:volgrowth}) and~(\ref{eq:riccati})
implies
\begin{align*}
\scal^S(\exp(ru)) &= (n-2)C+(v'(r)/v(r))^2+\Tr(\sigma'_u(r))+\Tr(R_{\dot\gamma_u(r)})\\
&= (n-2)C+(v'(r)/v(r))^2+(v'/v)'(r)+C = (n-1)C+v''(r)/v(r).
\end{align*}
Therefore, geodesic spheres in~$M$ have constant scalar curvature,
and the respective constant depends only on the radius, not on the midpoint.
Finally, using~(i) one concludes that
the function which associates to small values of~$r$ the scalar
curvature of geodesic spheres of radius~$r$ is determined already by the function
which associates to small values of~$r$ the volume of geodesic spheres of radius~$r$.
\end{remark}

\begin{remark}
\label{rem:damekricci}
As mentioned in the Introduction, the aim of this paper is to
show that in harmonic spaces, the heat invariants
$a_0(S_r(p))=\vol(S_r(p))$ and $a_2(S_r(p))$, viewed as functions of~$r$,
together determine $|\nabla R|^2_p$\,. This is not the case for $a_0$ alone,
as manifested by certain pairs of Damek-Ricci spaces. A Damek-Ricci space~$AN$
is a certain type of solvable Lie groups with left invariant metric, namely,
the standard $1$-dimensional solvable extension of a simply connected
Riemannian nilmanifold~$N$ of Heisenberg type. The volume density function
of~$AN$ is radial and depends only on the dimensions of~$N$ and its center~\cite{DR};
see also the book~\cite{BTV}.
Within the class of Damek-Ricci spaces, there exist pairs of symmetric
spaces~$AN$ and locally nonsymmetric spaces~$AN'$ where $N$ and~$N'$ have the
same dimension and so do their centers. (In fact, certain such pairs $AN$ and $AN'$ were
the ambient manifolds used by Szab\'o in \cite{Sz01}, \cite{Sz05}; recall
the Introduction.) In~particular, geodesic spheres of the same radius
in $AN$ and~$AN'$ have the same volume. This shows that in harmonic
spaces, the function $r\mapsto\vol(S_r(p))$ alone does not determine
$|\nabla R|^2_p$\,. In turn, Remark~\ref{rem:volscal} shows that in harmonic
spaces, the function $r\mapsto a_0(S_r(p))=\vol(S_r(p))$ already determines the
function $r\mapsto a_1(S_r(p))=\frac16\int_{S_r(p)}\scal^S\,\dvol_{S_r(p)}$\,.
Therefore, we need to consider $a_2(S_r(p))$. The next section
gives some necessary preparations for this.
\end{remark}

\Section{Radial expansions of $|\Ric|^2$ and $|R|^2$ for geodesic spheres in harmonic spaces}
\label{sec:curv}

\noindent
In this section we will describe a certain coefficient in the radial
power series expansions
of the curvature invariants $|\Ric|^2$ and $|R|^2$ of geodesic spheres in harmonic spaces.
First we need the following lemma.

\begin{lemma}
\label{lem:norms}
Let $M$ be an $n$-dimensional harmonic space, and let $C$ and $H$ be the
constants from Proposition~\ref{prop:constants}.
Let $p\in M$, and let $S:=S_r(p)$ be a geodesic
sphere around~$p$ with radius $r\in(0,i(p))$, endowed with the induced Riemannian
metric. Let $u$ be a unit vector in~$T_pM$, let $\sigma:=\sigma_u(r)$ be as in
Section~\ref{sec:prelim}, and write $\sigma':=\sigma'_u(r)$.
Let $R^S$ and $\Ric^S$ denote the curvature tensor, resp.~the Ricci operator, of~$S$.
Then in the point $q:=\exp(ru)\in S$ we have{\rm:}
\begin{align*}
{\rm(i)}&\
|\Ric^S|_q^2=(n-1)C^2+2C(\Tr(\sigma))^2+(\Tr(\sigma))^2\Tr(\sigma^2)
  +2C\Tr(\sigma')+2\Tr(\sigma)\Tr(\sigma\sigma')+\Tr(\sigma'\sigma'),\\
{\rm(ii)}&\
|R^S|_q^2=\frac23(n-4)\bigl((n+2)H-C^2\bigr)+4H+2(\Tr(\sigma^2))^2-2\Tr(\sigma^4)
+4\sum_{i=1}^n\Tr\bigl(\sigma\circ R(e_i,\,.\,) \sigma e_i\bigr),\\
&\ \mbox{where $\{e_1,\ldots,e_n\}$ is an orthonormal basis of $T_q M$}.
\end{align*}
\end{lemma}

\begin{proof}
(i)
Let $\nu$ be the outward pointing radial unit vector field as in Section~\ref{sec:prelim}.
From the Gauss equation one easily derives the following formula
whose analog is valid for submanifolds of codimension one in arbitrary
Riemannian manifolds:
$$\Ric^S_q=(\Ric-R_{\nu_q}+\Tr(\sigma)\sigma-\sigma^2)\restr{T_qS}
$$
Using the Einstein condition and the Riccati equation~(\ref{eq:riccati}),
this formula becomes in our situation:
$$\Ric^S_q=(C\Id+\Tr(\sigma)\sigma+\sigma')\restr{T_qS}
$$
(see also~\cite{NV}, p.~67).
Now one obtains the desired formula immediately, keeping in mind
that both~$\sigma$ and~$\sigma'$ are symmetric and annihilate $\nu_q$\,.

(ii)
Choose an orthonormal basis $\{e_1,\ldots,e_n\}$ of $T_q M$ such that
$e_1=\nu_q$. For all $i,j,k,\ell\in\{2,\ldots,n\}$ we have by the Gauss equation
(recall our sign convention for~$R$):
$$
\<R^S(e_i,e_j)e_k,e_\ell\>=\<R(e_i,e_j)e_k,e_\ell\>+\<\sigma e_i,e_k\>\<\sigma e_j,e_\ell\>
-\<\sigma e_j, e_k\>\<\sigma e_i,e_\ell\>.
$$
Squaring both sides and forming the sum over $i,j,k,\ell$, while recalling that
$\sigma$ is symmetric and annihilates~$e_1$, we get
\begin{align*}
|R^S|_q^2={}&\sum_{i,j,k,\ell=2}^n\<R(e_i,e_j)e_k,e_\ell\>^2
+|\sigma|^2|\sigma|^2+|\sigma|^2|\sigma|^2\\
&-2|\sigma^2|^2
+2\sum_{i,j=1}^n \<R(e_i,e_j)\sigma e_i,\sigma e_j\>
-2\sum_{i,j=1}^n \<R(e_i,e_j)\sigma e_j,\sigma e_i\>\\
={}&\sum_{i,j,k,\ell=2}^n \<R(e_i,e_j)e_k,e_\ell\>^2+2(\Tr(\sigma^2))^2-2\Tr(\sigma^4)
+4\sum_{i=1}^n\Tr\bigl(\sigma\circ R(e_i,\,.\,)\sigma e_i\bigr).
\end{align*}
The desired formula now follows from the fact that the first sum on the right hand
side is equal to
$|R|_q^2-4|R(e_1,\,.\,)\,.\,|^2+4|R_{e_1}|^2$
which by Proposition~\ref{prop:reqs}(i), (ii) and Proposition~\ref{prop:constants}(iii)
becomes $\frac23(n-4)\bigl((n+2)H-C^2\bigr)+4H$.
\end{proof}

Using the radial power series expansion of~$\sigma$ together
with the previous lemma, we will make conclusions concerning the
first few coefficients of the radial expansions of $|\Ric^S|^2$ and $|R^S|^2$.
The following proposition will be the key of the proof of the Main Theorem~\ref{thm:main}.
Actually we will use only the statements about $\alpha_2$ and $\beta_2$ in
this proposition.

\begin{proposition}
\label{prop:coeffs}
Let $M$ be an $n$-dimensional harmonic space, and let $C$, $H$, and~$L$
be the constants from Proposition~\ref{prop:constants}.
Let $p\in M$, and let $u$ be a unit vector in~$T_p M$. Then
\begin{align*}
|\Ric^{S_r(p)}|_{\exp(ru)}^2&=\alpha_{-4}r^{-4}+\alpha_{-2} r^{-2}+\alpha_0
+\alpha_2(u) r^2+O(r^3)\mbox{ and}\\
|R^{S_r(p)}|_{\exp(ru)}^2&=\beta_{-4}r^{-4}+\beta_{-2} r^{-2}+\beta_0+\beta_2(u) r^2+O(r^3)
\end{align*}
for $r\downarrow 0$, where the coefficients $\alpha_i$ and $\beta_i$
for $i\in \{-4,-2,0\}$ are constants depending only on $n$, $C$, and $H$.
Moreover,
\begin{align*}
\alpha_2(u)&=\hat\alpha_2+\frac1{16}\Tr(R'_uR'_u)\mbox{\ \ and}\\
\beta_2(u)&=\hat\beta_2
+\frac 49\sum_{i=1}^n\Tr\bigl(R_u\circ R(e_i,\,.\,)R_u e_i\bigr),
\end{align*}
where $\hat\alpha_2$ and $\hat\beta_2$ are constants depending only on $n$, $C$, $H$,
and $L$, and where $\{e_1,\ldots,e_n\}$ is an orthonormal basis of $T_p M$.
\end{proposition}

\begin{proof}
We use Lemma~\ref{lem:norms} together with the power series
expansions (\ref{eq:pow}),~(\ref{eq:powtrharm}) of $\sigma:=\sigma_u(r)$ and $\Tr(\sigma)$.
Let us first consider $|\Ric^{S_r(p)}|_{\exp(ru)}^2$ and the individual
contributions of the nonconstant terms in Lemma~\ref{lem:norms}(i) to its expansion.
By~(\ref{eq:powtrharm}) we have
\begin{equation*}
(\Tr(\sigma))^2=\bigl((n-1)\frac1r-\frac13Cr-\frac1{45}Hr^3-\frac1{15120}Lr^5\bigr)^2+O(r^6)
\end{equation*}
for $r\downarrow0$.
Moreover, from the expansion~(\ref{eq:pow}) and
Proposition~\ref{prop:constants} one gets
\begin{equation}
\label{eq:trsigmasq}
\Tr(\sigma^2)=(n-1)\frac1{r^2}-\frac23 C+\frac1{15}Hr^2+\frac1{3024}Lr^4+O(r^5).
\end{equation}
Further,
\begin{equation*}
\Tr(\sigma')=\frac d{dr}\Tr(\sigma)=-(n-1)\frac1{r^2}-\frac13C-\frac1{15}Hr^2+O(r^4)
\end{equation*}
by~(\ref{eq:powtrharm}), and
\begin{equation*}
2\Tr(\sigma\sigma')=\frac d{dr}\Tr(\sigma^2)=-2(n-1)\frac1{r^3}+\frac2{15}Hr+\frac1{756}Lr^3+O(r^4)
\end{equation*}
by~(\ref{eq:trsigmasq}). Using these expansions and~(\ref{eq:powtrharm}), one easily checks
that each of the the four individual terms $2C(\Tr(\sigma))^2$, $(\Tr(\sigma))^2\Tr(\sigma^2)$, $2C\Tr(\sigma')$,
and $2\Tr(\sigma)\Tr(\sigma\sigma')$ appearing on the right hand side of Lemma~\ref{lem:norms}(i)
has the property that the corresponding coefficients of $r^{-4}, r^{-2},r^0$ depend only on $n,C,H$,
the coefficient of~$r^2$ depends only on $n,C,H,L$, and the coefficients of $r^{-3},r^{-1},r$
vanish.

It remains to consider the term $\Tr(\sigma'\sigma')$ in Lemma~\ref{lem:norms}(i).
From~(\ref{eq:pow}) we get
\begin{equation*}
\begin{split}
P^{r,0}_{\gamma_u}\circ\sigma'\circ P^{0,r}_{\gamma_u}={}&-\frac1{r^2} I_u-\frac 13 R_u
-\frac r2 R'_u -\bigl(\frac3{10}R''_u+\frac1{15}R_uR_u\bigr)r^2\\
&-\bigl(\frac19R'''_u+\frac1{18}R_uR'_u
+\frac1{18}R'_uR_u\bigr)r^3\\
&-\bigl(\frac5{168}R^{(4)}_u+\frac1{42}R_uR''_u+\frac1{42}R''_uR_u+\frac5{112}R'_uR'_u
+\frac2{189}R_uR_uR_u\bigr)r^4\\
&+O(r^5)
\end{split}
\end{equation*}
and thereby, using Proposition~\ref{prop:constants}:
\begin{equation*}
\begin{split}
\Tr(\sigma'\sigma')={}&(n-1)\frac1{r^4}+\frac23\frac C{r^2}+\frac{11}{45}H\\
&+\Bigl(\bigl(-\frac2{21}+\frac5{56}-\frac15+\frac14\bigr)\Tr(R'_uR'_u)
+\bigl(\frac4{189}+\frac2{45}\bigr)\Tr(R_uR_uR_u)\Bigr)r^2+O(r^3).
\end{split}
\end{equation*}
The coefficient of~$r^2$ in the latter expansion is
$$\frac{37}{840}\Tr(R'_uR'_u)+\frac{62}{945}\Tr(R_uR_uR_u)
$$
which by Proposition~\ref{prop:constants}(vi) turns out to be
$$\frac{62}{32\cdot945}L+\bigl(\frac{37}{840}+\frac{9\cdot 62}{32\cdot945}\bigr)\Tr(R'_uR'_u)
=\frac{31}{15120}L+\frac1{16}\Tr(R'_uR'_u).
$$
This concludes the proof of the statements concerning the expansion of
$|\Ric^{S_r(p)}|^2_{\exp(ru)}$\,.

We now turn to $|R^{S_r(p)}|^2_{\exp(ru)}$ and study the individual
contributions of the nonconstant terms in Lemma~\ref{lem:norms}(ii) to its expansion.
Squaring~(\ref{eq:trsigmasq}), we see that in the expansion of the term
$2(\Tr(\sigma^2))^2$ the
coefficients of $r^{-4},r^{-2},r^0$ depend only on $n,C,H$, the coefficient of~$r^2$
depends only on $n,C,H,L$, and the coefficients of~$r^{-3},r^{-1},r$ vanish.

Regarding the term $-2\Tr(\sigma^4)$ we obtain from~(\ref{eq:pow}):
\begin{multline*}
P^{r,0}_{\gamma_u}\circ\sigma^4\circ P^{0,r}_{\gamma_u}=
\frac1{r^4}I_u-\frac 4{3r^2} R_u-\frac 1r R'_u +\bigl(-\frac25R''_u+\frac{26}{45}R_u^2\bigr)
+\bigl(-\frac 19 R'''_u+\frac 49 (R_uR'_u+R'_uR_u)\bigr)r\\
+\Bigl(-\frac1{42}R^{(4)}_u +\bigl(-\frac 2{105}+\frac 15\bigr)(R''_uR_u+R_uR''_u)
+\bigl(-\frac 1{28}+\frac 38\bigr)R'_uR'_u
+\bigl(-\frac 8{945}+\frac4{45}-\frac 4{27}\bigr)R_u^3\Bigr)r^2\\
+O(r^3)
\end{multline*}
for $r\downarrow0$. Using Proposition~\ref{prop:constants} we get
\begin{align*}
-2\Tr(\sigma^4)={}&-2(n-1)\frac1{r^4}+\frac 83 \frac C{r^2}-\frac{52}{45}H\\
&+\Bigl(\bigl(-\frac8{105}+\frac 45+\frac1{14}-\frac 34\bigr)\Tr(R'_uR'_u)
+\bigl(\frac{16}{945}-\frac8{45}+\frac8{27}\bigr)\Tr(R_uR_uR_u)\Bigr)r^2+O(r^3).
\end{align*}
The coefficient of~$r^2$ in the latter expansion is
$$
\frac{19}{420}\Tr(R'_uR'_u)+\frac{128}{945}\Tr(R_uR_uR_u)
$$
which by Proposition~\ref{prop:constants}(vi) equals
\begin{equation}
\label{eq:1over12}
\frac{128}{32\cdot 945}L+\bigl(\frac{19}{420}+\frac{9\cdot128}{32\cdot945}\bigr)
\Tr(R'_uR'_u)=\frac4{945}L+\frac1{12}\Tr(R'_uR'_u).
\end{equation}

It remains to consider the term
$4\sum_{i=1}^n\Tr\bigl(\sigma\circ R(e_i,\,.\,)\sigma e_i\bigr)$
in Lemma~\ref{lem:norms}(ii).
We make some preliminary observations.
For $k\in\N_0$, let $R^{(k)}$, resp.~$\Ric^{(k)}$ denote the
$k$-th covariant derivative of the
curvature tensor, resp.~the Ricci operator,
along~$\gamma_u$ at $r=0$. We will use the
the Taylor series expansion of the Riemannian curvature tensor
along~$\gamma_u$ (recall that $M$ is analytic):
\begin{equation}
\label{eq:taylorR}
P_{\gamma_u}^{r,0}\circ R_{\gamma_u(r)}\circ P_{\gamma_u}^{0,r}=
\sum_{k=0}^\infty\frac{r^k}{k!}R^{(k)}
\end{equation}
Moreover, $\Ric^{(k)}=0$ for $k\ge1$ since
$M$ is Einstein.
Note that $\Ric\restr{T_pM}=\sum_{i=1}^n R_{e_i}$ and
similarly on each $T_{\gamma_u(r)}M$ if we extend $\{e_1,\ldots,e_n\}$
parallelly along~$\gamma_u$.
For any $k\in\N_0$ we have, using Proposition~\ref{prop:constants}:
\begin{equation}
\label{eq:IuIu}
\begin{split}
\sum_{i=1}^n\Tr\bigl(I_u\circ R^{(k)}(e_i,\,.\,)I_ue_i\bigr)
&=\Tr(I_u\circ\Ric^{(k)})-\Tr(I_u\circ R^{(k)}_u)\\
&=\Tr(\Ric^{(k)})-\<\Ric^{(k)}u,u\>-\Tr(R^{(k)}_u)
=\begin{cases} (n-2)C,& k=0,\\0,& k\ge1.
\end{cases}
\end{split}
\end{equation}
Moreover,
\begin{equation}
\begin{split}
\label{eq:RuIu}
\sum_{i=1}^n\Tr\bigl(R_u\circ R^{(k)}(e_i,\,.\,)I_ue_i\bigr)
&=\Tr(R_u\circ\Ric^{(k)})-\Tr(R_uR^{(k)}_u)\\
&=\begin{cases} C^2-H,&k=0,\\-\Tr(R_uR^{(k)}_u),&k\ge1,
\end{cases}
\end{split}
\end{equation}

\begin{equation}
\begin{split}
\label{eq:RpuIu}
\sum_{i=1}^n\Tr\bigl(R'_u\circ R^{(k)}(e_i,\,.\,)I_ue_i\bigr)
&=\Tr(R'_u\circ\Ric^{(k)})-\Tr(R'_uR^{(k)}_u)\\
&=\begin{cases}0,&k=0,\\-\Tr(R'_uR^{(k)}_u),&k\ge1,
\end{cases}
\end{split}
\end{equation}

\begin{align}
\label{eq:RppuIu}
\sum_{i=1}^n\Tr\bigl(R''_u\circ R(e_i,\,.\,)I_ue_i\bigr)
&=\Tr(R''_u\circ\Ric)-\Tr(R''_uR_u)=0+\Tr(R'_uR'_u),
\\
\label{eq:RuRu_Iu}
\sum_{i=1}^n\Tr\bigl(R_uR_u\circ R(e_i,\,.\,)I_ue_i\bigr)
&=\Tr(R_uR_u\circ\Ric)-\Tr(R_uR_uR_u)=CH-\Tr(R_uR_uR_u).
\end{align}
Note that for any pair of symmetric endomorphisms $F,G$ of $T_pM$
we have
\begin{equation}
\label{eq:FG}
\sum_{i=1}^n\Tr\bigl(F\circ R^{(k)}(e_i,\,.\,)Ge_i\bigr)
=\sum_{i=1}^n\Tr\bigl(G\circ R^{(k)}(e_i,\,.\,)Fe_i\bigr)
\end{equation}
by the symmetries of the curvature operator.
Keeping the expansions (\ref{eq:pow}) and~(\ref{eq:taylorR}) in mind,
we see that the expression in~(\ref{eq:IuIu}) contributes
only to the coefficient of~$r^{-2}$ in the expansion
of $\sum_{i=1}^n\Tr\bigl(\sigma\circ R(e_i,\,.\,)\sigma e_i\bigr)$,
the expression in (\ref{eq:RuIu}) contributes to
the coefficients of~$r^0$ and $r^2$ (and higher order),
the expressions in (\ref{eq:RpuIu}), (\ref{eq:RppuIu}), (\ref{eq:RuRu_Iu})
contribute to the coefficient of~$r^2$ (and higher order).
The only additional contribution
to the coefficient of~$r^2$ is given by the sum
of $\Tr\bigl(R_u\circ R(e_i,\,.\,)R_ue_i\bigr)$.
Recalling~(\ref{eq:FG}) (and multiplying $R^{(k)}$ by $1/k!$), we obtain from
(\ref{eq:pow}), \ref{prop:constants}(iv), (\ref{eq:taylorR}), and the above observations:
\begin{equation*}
\begin{split}
4\sum_{i=1}^n\Tr\bigl(\sigma\circ R(e_i,\,.\,)\sigma e_i\bigr)
={}&4\Bigl((n-2)\frac C{r^2} -\frac 23 (C^2-H)\\
&+\Bigl[\frac 2{3\cdot2!}\Tr(R_uR''_u)+\bigl(\frac 24-\frac 2{10}\bigr)\Tr(R'_uR'_u)
-\frac2{45}CH+\frac2{45}\Tr(R_uR_uR_u)\\
&\hphantom{\Big[\frac13\Tr}+\frac19\sum_{i=1}^n\Tr\bigl(R_u\circ R(e_i,\,.\,)R_ue_i\bigr)\Bigr]r^2\Bigr)+O(r^3).
\end{split}
\end{equation*}
By Proposition~\ref{prop:constants}(v), the coefficient of~$r^2$ in the latter expansion is
$$
-\frac8{45}CH-\frac2{15}\Tr(R'_uR'_u)+\frac8{45}\Tr(R_uR_uR_u)+\frac49
\sum_{i=1}^n\Tr\bigl(R_u\circ R(e_i,\,.\,)R_ue_i\bigr)
$$
By Proposition~\ref{prop:constants}(vi), the two
terms involving $\Tr(R'_uR'_u)$ and $\Tr(R_uR_uR_u)$ become
$$
\frac8{32\cdot 45}L+\bigl(-\frac2{15}+\frac{9\cdot 8}{32\cdot 45}\bigr)\Tr(R'_uR'_u)
=\frac1{180}L-\frac1{12}\Tr(R'_uR'_u).
$$
Combining this with the result for the $r^2$-coefficient of $-2\Tr(\sigma^4)$
from~(\ref{eq:1over12}), we conclude that the terms involving
$\Tr(R'_uR'_u)$ in the coefficient of $r^2$ in the power series expansion of~$|R^{S_r(p)}|_{\exp(ru)}^2$
cancel each other, and the only remaining
term apart from those which depend solely on $n,C,H,L$ is
$\frac49\sum_{i=1}^n\Tr\bigl(R_u\circ R(e_i,\,.\,)R_ue_i\bigr)$, as claimed.
\end{proof}

\begin{remark}
\label{rem:tracesv}
For the purpose of the proof of the Main Theorem~\ref{thm:main} in Section~\ref{sec:proof},
which we will perform using the heat invariants $a_0(S_r(p))=\vol(S_r(p))$ and $a_2(S_r(p))
=\frac1{360}\int_{S_r(p)}(5(\scal^S)^2-2|\Ric^S|^2+2|R^S|^2)\dvol_{S_r(p)}$,
we would actually not have needed the exact statement of the previous proposition
 -- which might, however, be interesting in its own right.
Rather, we could have restricted our attention to the term $\Tr(\sigma'\sigma')$
in the expression of~$|\Ric^S|_{\gamma_u(r)}^2$ in Lemma~\ref{lem:norms}(i),
and to the last two terms in the expression of~$|R^S|_{\gamma_u(r)}^2$
in Lemma~\ref{lem:norms}(ii).
In fact, even without the explicit calculation of the expansion of the other
terms, one easily sees that those are determined by the volume function
$r\mapsto a_0(S_r(p))=\vol(S_r(p))$
of the geodesic spheres (which is just the function~$v$ multiplied by the volume
of the standard unit sphere, see Remark~\ref{rem:ave} below).
More precisely, in the spirit of Remark~\ref{rem:volscal} we obtain
\begin{align*}
2C(\Tr(\sigma))^2&=2C(v'/v)^2,\\
(\Tr(\sigma))^2\Tr(\sigma^2)&=(v'/v)^2(-(v'/v)'-C),\\
2C\Tr(\sigma')&=2C\Tr(\sigma)'=2C(v'/v)',\\
2\Tr(\sigma)\Tr(\sigma\sigma')&=2v'/v\cdot\tfrac12(\Tr(\sigma^2))'
 =v'/v\cdot(-(v'/v)''),\\
2(\Tr(\sigma^2))^2&=2(-(v'/v)'-C)^2.
\end{align*}
\end{remark}

\Section{Proof of the Main Theorem}
\label{sec:proof}

\noindent
In this section we will first derive an integrated version
of Proposition~\ref{prop:coeffs}. Using this and the heat invariants
$a_0,a_1,a_2$ of geodesic spheres in harmonic spaces we will then
prove our Main Theorem~\ref{thm:main}.
We need the following general remark on mean values.

\begin{remark}
\label{rem:ave}
In any harmonic space~$M$, the average (or mean value)
of a smooth function~$f$ on a geodesic sphere $S_r(p)$ (with
$0<r<i(p)$) is the same as the average of $f(\exp(r\,.\,))$ over
the unit sphere $S_1(0_p)$ in~$T_pM$. More explicitly:
Let $\omega_{n-1}$ denote the volume of the $(n-1)$-dimensional
standard sphere. In particular, $\omega_{n-1}$ is the volume
of $S_1(0_p)$.
Recall from Section~\ref{sec:prelim} that $\theta(r)=\theta_u(r)$
is independent of~$u$ (and even of~$p$) by harmonicity.
We have
\begin{equation*}
\vol(S_r(p))=r^{n-1}\theta(r)\omega_{n-1}=v(r)\omega_{n-1}\,,
\end{equation*}
and for any smooth function~$f$ on $S_r(p)$,
\begin{equation*}
\begin{split}
\frac1{\vol(S_r(p))}\int_{S_r(p)}f\,\dvol_{S_r(p)}
&=\frac1{v(r)\omega_{n-1}}
\int_{S_1(0_p)}f(\exp(ru))v(r)\,du\\
&=\frac1{\omega_{n-1}}\int_{S_1(0_p)}f(\exp(ru))\,du.
\end{split}
\end{equation*}
\end{remark}

Now we can give an ``integrated'' version of Proposition~\ref{prop:coeffs}.

\begin{proposition}
\label{prop:intcoeffs}
Let $M$ be an $n$-dimensional harmonic space, and let $C$, $H$, and~$L$
be the constants from Proposition~\ref{prop:constants}.
Let $p\in M$. Then
\begin{align*}
\frac1{\vol(S_r(p))}\int_{S_r(p)} |\Ric^{S_r(p)}|^2 \dvol_{S_r(p)}&=
\alpha_{-4}r^{-4}+\alpha_{-2} r^{-2}+\alpha_0
+\overline\alpha_2 r^2+O(r^3)\mbox{ and}\\
\frac1{\vol(S_r(p))}\int_{S_r(p)} |R^{S_r(p)}|^2 \dvol_{S_r(p)}&=
\beta_{-4}r^{-4}+\beta_{-2} r^{-2}+\beta_0+\overline\beta_2 r^2+O(r^3)
\end{align*}
for $r\downarrow 0$, where the coefficients $\alpha_i$ and $\beta_i$
for $i\in \{-4,-2,0\}$ are the constants from Proposition~\ref{prop:coeffs}
depending only on $n$, $C$, and $H$.
Moreover,
\begin{align*}
\overline\alpha_2&=\tilde\alpha_2+\frac3{16n(n+2)(n+4)}|\nabla R|_p^2\mbox{\ \ and}\\
\overline\beta_2&=\tilde\beta_2
+\frac1{8n(n+2)}|\nabla R|_p^2,
\end{align*}
where $\tilde\alpha_2$ and $\tilde\beta_2$ are constants
depending only on $n$, $C$, $H$, and $L$.
\end{proposition}

\begin{proof}
For any unit vector $u$ in $T_pM$, let $\alpha_2(u)$ and $\beta_2(u)$
be the coefficients from Proposition~\ref{prop:coeffs}. Using that
proposition and Remark~\ref{rem:ave}, we only need to show that
\begin{align*}
\overline\alpha_2&:=\frac1{\omega_{n-1}}\int_{S_1(0_p)}\alpha_2(u)\,du
=\hat\alpha_2+\frac1{\omega_{n-1}}\int_{S_1(0_p)}\frac1{16}\Tr(R'_uR'_u)\,du\mbox{\ \ and}
\\
\overline\beta_2&:=\frac1{\omega_{n-1}}\int_{S_1(0_p)}\beta_2(u)\,du
=\hat\beta_2+\frac1{\omega_{n-1}}\int_{S_1(0_p)}\frac 49\sum_{i=1}^n
\Tr\bigl(R_u\circ R(e_i,\,.\,)R_ue_i\bigr)\,du
\end{align*}
are of the claimed form, where $\hat\alpha_2,\hat\beta_2$ are
as in Proposition~\ref{prop:coeffs}.
For $\overline\alpha_2$ this follows immediately (with $\tilde\alpha_2:=\hat\alpha_2$)
from the following formula
(see the proof of Theorem~5.7 of~\cite{NV}; details of the computation
can be found on p.~170 of~\cite{GV}):
\begin{equation}
\label{eq:int16}
\int_{S_1(0_p)}\Tr(R'_uR'_u)\,du=\frac{3\omega_{n-1}}{n(n+2)(n+4)}|\nabla R|_p^2
\end{equation}
This confirms the statement concerning~$\overline\alpha_2$.

We now consider $\overline\beta_2$.
Writing $u=\sum_{i=1}^n u_ie_i$ and $R_{ijk\ell}=\<R(e_i,e_j)e_k,e_\ell\>$
we have
\begin{equation}
\begin{split}
\label{eq:int49}
\sum_{i=1}^n\Tr\bigl(R_u\circ R(e_i,\,.\,)R_u e_i\bigr)
&=\sum_{i,j,k,\ell=1}^n \<R(e_i,e_j)e_k,e_\ell\>\<R_ue_i,e_k\>\<R_ue_j,e_\ell\>\\
&=\sum_{a,b,c,d=1}^n\Bigl[\sum_{i,j,k,\ell=1}^n R_{ijk\ell} R_{aibk} R_{cjd\ell}\Bigr]
 u_au_bu_cu_d.
\end{split}
\end{equation}
Note that the integral of $u_au_bu_cu_d$ over $S_1(0_p)$ is zero whenever
$\{a,b,c,d\}$ contains at least three different elements.
Abbreviating $A_{abcd}:=\sum_{i,j,k,\ell=1}^n R_{ijk\ell} R_{aibk} R_{cjd\ell}$
we have, using the Einstein condition and recalling the definition of $\hat R$
and~$\rcirc$ from Section~\ref{sec:prelim}:
\begin{align*}
\sum_{a,b=1}^n A_{aabb}&=\sum_{a,b,i,j,k,\ell=1}^n R_{ijk\ell} R_{aiak} R_{bjb\ell}
= C^2\sum_{i,j,k,\ell=1}^n R_{ijk\ell} \delta_{ik} \delta_{j\ell}
= C^2\sum_{i,j=1}^n R_{ijij} = nC^3,\\
\sum_{a,b=1}^n A_{abab}&=\sum_{a,b,i,j,k,\ell=1}^n R_{ijk\ell} R_{aibk} R_{ajb\ell}
= \sum_{a,b,i,j,k,\ell=1}^n R_{ijk\ell} R_{aibk} R_{ja\ell b}
= \rcirc(p),\\
\sum_{a,b=1}^n A_{abba}&=\sum_{a,b,i,j,k,\ell=1}^n R_{ijk\ell} R_{aibk} R_{bja\ell}
=\sum_{a,b,i,j,k,\ell=1}^n R_{ijk\ell} R_{aibk} R_{jb\ell a} = \rcirc(p)-\frac14\hat R(p),
\end{align*}
where for the last equality we have used formula~(2.7)(vi) of~\cite{Sa}; see also
formula~(2.15) of~\cite{GV}. Let $S^{n-1}\subset R^n$ denote the $(n-1)$-dimensional
standard sphere. Note that $\int_{S^{n-1}}u_1^2u_2^2\,du=\frac{\omega_{n-1}}{n(n+2)}$
and $\int_{S^{n-1}}u_1^4\,du=\frac{3\omega_{n-1}}{n(n+2)}$.
From the above equations and~(\ref{eq:int49}) we thus obtain
\begin{align*}
\int_{S_1(0_p)}\sum_{i=1}^n&\Tr\bigl(R_u\circ R(e_i,\,.\,)R_u e_i\bigr)\,du\\
&=\sum_{ \genfrac{}{}{0pt}{} {a,b\in\{1,\ldots,n\}} {a\ne b} }
[A_{aabb}+A_{abab}+A_{abba}]\int_{S^{n-1}}
u_1^2u_2^2\,du+
\sum_{a=1}^n A_{aaaa}\int_{S^{n-1}}u_1^4\,du\\
&=\sum_{a,b=1}^n [A_{aabb}+A_{abab}+A_{abba}]\frac{\omega_{n-1}}{n(n+2)}
=\bigl(nC^3+2\rcirc(p)-\frac14\hat R(p)\bigr)\frac{\omega_{n-1}}{n(n+2)}
\end{align*}
Hence,
\begin{equation}
\label{eq:b2}
\overline\beta_2=\hat\beta_2+\frac49\bigl(nC^3+2\rcirc(p)-\frac14\hat R(p)\bigr)\frac1{n(n+2)}
=\hat\beta_2+\frac{4C^3}{9(n+2)}+\bigl(\frac89\rcirc(p)-\frac19\hat R(p)\bigr)\frac1{n(n+2)}
\end{equation}
Recall from Proposition~\ref{prop:reqs}(ii), (iii) and equation~(\ref{eq:lichn}) that
\begin{align*}
112\hat R(p)-32\rcirc(p) &= 27|\nabla R|_p^2\,+\,\mbox{some constant depending only on }n,C,H,L\mbox{ and }\\
\hphantom{112}\hat R(p)+\hphantom{32}\llap{4}\rcirc(p)
&=\hphantom{27} |\nabla R|_p^2\,+\,\mbox{some constant depending only on }n,C,H,
\end{align*}
using which one easily computes that
\begin{equation*}
-\frac19\hat R(p)+\frac89\rcirc(p) =
\frac18 |\nabla R|_p^2\,+\,\mbox{some constant depending only on }n,C,H,L.
\end{equation*}
Thus we conclude from~(\ref{eq:b2}):
$$
\overline\beta_2=\tilde\beta_2+\frac1{8n(n+2)}|\nabla R|_p^2\,,
$$
where $\tilde\beta_2$ is a constant depending only on $n,C,H,L$.
\end{proof}

\medskip
\noindent
{\bf Proof of the Main Theorem~\ref{thm:main}:}

\noindent
Let $M_1$, $M_2$ be harmonic spaces, $p_1\in M_1$, $p_2\in M_2$, and assume
there exists $\eps$ in the interval $(0,\min\{i(p_1),i(p_2)\})$ such that for each $0<r<\eps$ the geodesic spheres
$S_r(p_1)$ and~$S_r(p_2)$ are isospectral. Then $\dim M_1=:n=\dim M_2$, and the heat
invariants of the geodesic spheres coincide:
$$a_k(S_r(p_1))=a_k(S_r(p_2))
$$
for each $r\in(0,\eps)$ and all $k\in\N_0$.
We want to deduce that $|\nabla R|_{p_1}^2=|\nabla R|_{p_2}^2$.
Actually this will follow using just $a_0$ and~$a_2$.

Reformulating the problem, let $M$ be an $n$-dimensional harmonic space and $p\in M$.
We want to show that for any $\eps\in(0,i(p))$, the two functions
$$\phi_k:(0,\eps)\ni r\mapsto a_k(S_r(p))\in\R
$$
with $k\in\{0,2\}$ together determine the value of $|\nabla R|_p^2$\,.
By Remark~\ref{rem:volscal}(i), the function
$$\phi_0:r\mapsto a_0(S_r(p))=\vol(S_r(p))=v(r)\omega_{n-1}
$$
determines the constants $C,H,L$ associated with~$M$ (see Section~\ref{sec:prelim}).
Recall that the scalar curvature $\scal^S=:\scal^{S_r}$ of $S_r(p)$ is constant
on the manifold $S_r(p)$, and that the 
function $v:(0,\eps)\to\R$ determines, by Remark~\ref{rem:volscal}(ii),
the function $(0,\eps)\ni r\mapsto\scal^{S_r}\in\R$.
In particular,
the function $\phi_0=v\omega_{n-1}$
also determines the function
$(0,\eps)\ni r\mapsto\int_{S_r(p)}(\scal^S)^2\dvol_{S_r(p)}=\phi_0(r)\cdot(\scal^{S_r})^2\in\R$.
By
\begin{align*}
\phi_2(r)=a_2(S_r(p))&=\frac1{360}\int_{S_r(p)}\bigl
(5(\scal^S)^2-2|\Ric^S|^2+2|R^S|^2\bigr)\dvol_{S_r(p)}
\end{align*}
it follows that the functions $\phi_0$ and $\phi_2$ together determine
the function
$$(0,\eps)\ni r\mapsto\frac1{\vol(S_r(p))}\int_{S_r(p)}\bigl(|R^S|^2-|\Ric^S|^2\bigr)
\dvol_{S_r(p)}\in\R.
$$
By Proposition~\ref{prop:intcoeffs}, the $r^2$-coefficient in the
power series expansion of this function is the sum of the term
$$\Bigl(\frac1{8n(n+2)}-\frac3{16n(n+2)(n+4)}\Bigr)|\nabla R|_p^2
=\frac{2n+5}{16n(n+2)(n+4)}|\nabla R|_p^2
$$
and $\tilde\beta_2-\tilde\alpha_2$.
Recall that the latter is a constant depending only on
$n,C,H,L$, and is thus determined by~$\phi_0$. We conclude that the functions
$\phi_0$ and $\phi_2$ together determine $|\nabla R|_p^2$\,, as claimed.
\hfill\qed

\Section{Geodesic balls}
\label{sec:balls}

\noindent
In this section we will prove the following version of the Main Theorem~\ref{thm:main}
for geodesic balls:

\begin{theorem}
\label{thm:balls}
Let $M_1$ and $M_2$ be harmonic spaces, and let $p_1\in M_1$,
$p_2\in M_2$. If there exists $\eps>0$ such that for each
$r\in(0,\eps)$ the geodesic balls $B_r(p_1)$ and $B_r(p_2)$
are Dirichlet isospectral, then $|\nabla R|_{p_1}^2=|\nabla R|_{p_2}^2$.
The same holds if the assumption of Dirichlet isospectrality
is replaced by the assumption of Neumann isospectrality.
\end{theorem}

This theorem implies the corresponding analog of our Main Corollary~\ref{cor:main}:

\begin{corollary}
\label{cor:balls}
Let $M_1$ and $M_2$ be harmonic spaces. Assume that the Dirichlet
isospectrality
hypothesis of Theorem~\ref{thm:balls} is satisfied for
\emph{each} pair of points $p_1\in M_1$, $p_2\in M_2$.
Then $M_1$ is locally symmetric if and only if $M_2$ is locally
symmetric. The same holds if the assumption of Dirichlet isospectrality
is replaced by the assumption of Neumann isospectrality.
\end{corollary}

For the proof of Theorem~\ref{thm:balls}
we will use the heat invariants for manifolds with boundary.
Let $M$ be an $n$-dimensional Riemannian manifold, and let $B\subset M$
be a compact domain with smooth boundary.
If $\Delta$ denotes the Laplace operator on~$B$ with Dirichlet boundary
conditions then there is an asymptotic expansion
$$
\Tr(\exp(-t\Delta))\sim(4\pi t)^{-n/2}\sum_k a_k^D(B) t^k
$$
for $t\downarrow0$,
where $k=0, 0.5, 1, 1.5, \ldots$ ranges over the nonnegative half
integers (see~\cite{BG}). For the Laplace operator on~$B$ with Neumann boundary
conditions the analog of this formula holds with certain coefficients
$a_k^N(B)$. The coefficients $a_k^D(B)$ (resp.~$a_k^N(B)$) are given by
certain curvature integrals over $B$ and~$\partial B$. One has
$a_0^D(B)=a_0^N(B)=\vol(B)$ and $a_{0.5}^D(B)=-a_{0.5}^N(B)=-\frac{\sqrt{\pi}}2
\vol(\partial B)$ (see~\cite{BG}). In the proof of Theorem~\ref{thm:balls} we will use
the explicit formulas for $a_2^D(B)$ and $a_2^N(B)$ from~\cite{BG}.
Let $\nu$ denote the outward pointing
unit vector field on the boundary~$\partial B$ of~$B$,
and let $\sigma=\nabla\nu$ be the associated
shape operator.
Let $\scal$, $\Ric$, $R$ always refer to the usual
objects on~$M$ (not to the ones associated with the induced metric on~$\partial B$). Then
\begin{align*}
a_2^D(B)&=\frac1{360}\biggl[\int_B \bigl(-12\Delta(\scal)+5\scal^2-2|\Ric|^2+2|R|^2\bigr)\dvol_B\\
&+\int_{\partial B}\Bigl(18\nu(\scal)+20\scal\cdot\Tr(\sigma)
-4\Tr(R_\nu)\Tr(\sigma)+12\Tr(R_\nu\circ\sigma)\\
&\hphantom{{}+\int_{\partial B}\Bigl(}-4\Tr((\Ric-R_\nu)\circ\sigma)
+\frac{40}{21}(\Tr(\sigma))^3-\frac{88}7\Tr(\sigma)\Tr(\sigma^2)
+\frac{320}{21}\Tr(\sigma^3)\Bigr)\dvol_{\partial B}\biggr],
\end{align*}
\begin{align*}
a_2^N(B)&=\frac1{360}
\biggl[\int_B \bigl(-12\Delta(\scal)+5\scal^2-2|\Ric|^2+2|R|^2\bigr)\dvol_B\\
&+\int_{\partial B}\Bigl(-42\nu(\scal)+20\scal\cdot\Tr(\sigma)
-4\Tr(R_\nu)\Tr(\sigma)+12\Tr(R_\nu\circ\sigma)\\
&\hphantom{{}+\int_{\partial B}\Bigl(}-4\Tr((\Ric-R_\nu)\circ\sigma)
+\frac{40}{3}(\Tr(\sigma))^3+8\Tr(\sigma)\Tr(\sigma^2)
+\frac{32}3\Tr(\sigma^3)\Bigr)\dvol_{\partial B}\biggr].
\end{align*}

If $M$ is harmonic then, by the results of Section~\ref{sec:prelim},
the previous formulas simplify to
\begin{equation}
\label{eq:a2d}
\begin{split}
a_2^D(B)={}&\frac1{360}\Bigl[\vol(B)\cdot\bigl(5(nC)^2-2nC^2+\frac43n((n+2)H-C^2)\bigr)\\
&+\int_{\partial B}\Bigl(20nC\Tr(\sigma)
-8C\Tr(\sigma)+16\Tr(R_\nu\circ\sigma)\\
&\hphantom{{}+\int_{\partial B}\bigl(}
+\frac{40}{21}(\Tr(\sigma))^3-\frac{88}7\Tr(\sigma)\Tr(\sigma^2)
+\frac{320}{21}\Tr(\sigma^3)\Bigr)\dvol_{\partial B}\Bigr],
\end{split}
\end{equation}
\begin{equation}
\begin{split}
\label{eq:a2n}
a_2^N(B)&=\frac1{360}
\Bigl[\vol(B)\cdot\bigl(5(nC)^2-2nC^2+\frac43n((n+2)H-C^2)\bigr)\\
&+\int_{\partial B}\Bigl(20nC\Tr(\sigma)
-8C\Tr(\sigma)+16\Tr(R_\nu\circ\sigma)\\
&\hphantom{{}+\int_{\partial B}\bigl(}
+\frac{40}3(\Tr(\sigma))^3+8\Tr(\sigma)\Tr(\sigma^2)
+\frac{32}3\Tr(\sigma^3)\Bigr)\dvol_{\partial B}\Bigr].
\end{split}
\end{equation}

In the proof of Theorem~\ref{thm:balls} we will follow a similar strategy
as in the proof of Theorem~\ref{thm:main}. To this end, we need some preliminary
results
in the special case that $B=B_r(p)$ with $r\in(0,i(p))$ and $M$ is harmonic.
We remark -- without going into details this time -- that one can compute in
this case,
using equation~(\ref{eq:pow}),
Proposition~\ref{prop:constants},
and the Taylor series expansion
$$
P_{\gamma_u}^{r,0}\circ(R_\nu)_{\gamma_u(r)}\circ P_{\gamma_u}^{0,r}=\sum_{k=0}^\infty
\frac{r^k}{k!}R^{(k)}_u,
$$
that the $r^3$-coefficient
in the power series expansion of $\Tr(R_\nu\circ\sigma)$ equals
$-\frac1{1440}L+\frac1{96}\Tr(R'_uR'_u)$,
and that the $r^3$-coefficient in the power series expansion of $\Tr(\sigma^3)$
equals
$\frac1{30240}L-\frac1{96}\Tr(R'_uR'_u)$.
(That the contributions of $\Tr(R'_uR'_u)$ in these terms are negatives
of each other can also be checked as follows: Using~(\ref{eq:riccati}) twice,
we have $\Tr(R_\nu\circ\sigma)+\Tr(\sigma^3)=
-\Tr(\sigma'\sigma)=-\frac12\Tr(\sigma^2)'=\frac12\Tr(R_\nu+\sigma')'
=\frac12\Tr(\sigma)''$ whose $r^3$-coefficient
indeed depends only on~$L$ by~(\ref{eq:powtrharm}).)
Using~(\ref{eq:int16}), we conclude that
the $r^3$-coefficient in the power series expansion of
$r\mapsto\frac1{vol(S_r(p))}\int_{S_r(p)}\Tr(R_\nu\circ\sigma)\dvol_{S_r(p)}$
is
\begin{equation}
\label{rsigcoeff}
-\frac1{1440}L+\frac1{32n(n+2)(n+4)}|\nabla R|_p^2\,,
\end{equation}
Similarly, the $r^3$-coefficient in the power series expansion of
$r\mapsto\frac1{vol(S_r(p))}\int_{S_r(p)}\Tr(\sigma^3)\dvol_{S_r(p)}$
is
\begin{equation}
\label{sig3coeff}
\frac1{30240}L-\frac1{32n(n+2)(n+4)}|\nabla R|_p^2\,.
\end{equation}

\medskip
\noindent
{\bf Proof of Theorem~\ref{thm:balls}:}

\noindent
Let $M_1$, $M_2$ be harmonic spaces, $p_1\in M_1$, $p_2\in M_2$, and assume
there exists $\eps$ in the interval $(0,\min\{i(p_1),i(p_2)\})$
such that for each $0<r<\eps$ the geodesic spheres
$B_r(p_1)$ and~$B_r(p_2)$ are Dirichlet isospectral (resp.~Neumann isospectral).
Then $\dim M_1=:n=\dim M_2$, and the heat
invariants of the geodesic spheres coincide:
$$a^D_k(B_r(p_1))=a^D_k(B_r(p_2)),\mbox{\ \ resp.\ \ }a^N_k(B_r(p_1))=a^N_k(B_r(p_2))
$$
for each $r\in(0,\eps)$ and all $k\in\N_0$.
We want to deduce that $|\nabla R|_{p_1}^2=|\nabla R|_{p_2}^2$.
Actually this will follow using just $a_0$ and~$a_2$.
(We remark without proof here that, viewed as functions of~$r$, the
heat coefficients $a_{0.5}$, $a_1$, and $a_{1.5}$
do actually not contain more information than $a_0$ in our situation.)

We first consider the case of Dirichlet conditions.
Similarly as in the proof of Theorem~\ref{thm:main},
we reformulate the problem as follows:
Let $M$ be an $n$-dimensional harmonic space and $p\in M$.
We want to show that for any $\eps\in(0,i(p))$, the two functions
$$\psi^D_k:(0,\eps)\ni r\mapsto a^D_k(B_r(p))\in\R
$$
with $k\in\{0,2\}$ together determine the value of $|\nabla R|_p^2$\,.
Note that the function
$$\psi^D_0:r\mapsto a^D_0(B_r(p))=\vol(B_r(p))
$$
determines its own derivative which is just
$$r\mapsto\vol(S_r(p))=v(r)\omega_{n-1}
$$
(see the previous section). By Remark~\ref{rem:volscal}(i), we conclude
that $\psi^D_0$ again determines the constants $C,H,L$ associated with~$M$.
Moreover, the function $v:(0,\eps)\to\R$ determines the radial
functions $\Tr(\sigma)=v'/v$ and $\Tr(\sigma^2)=-(v'/v)'-C$ (compare
Remark~\ref{rem:tracesv}).
By~(\ref{eq:a2d}) it now follows that $\psi^D_0$ and $\psi^D_2$ together
determine the function
$$(0,\eps)\ni r\mapsto\frac1{\vol(S_r(p))}\int_{S_r(p)}\bigl(
16\Tr(R_\nu\circ\sigma)+\frac{320}{21}\Tr(\sigma^3)\bigr)
\dvol_{S_r(p)}\in\R.
$$
Recalling (\ref{rsigcoeff}) and~(\ref{sig3coeff}), we see that
the $r^3$-coefficient in the
power series expansion of the latter function is the sum of
$$
\frac1{32n(n+2)(n+4)}\bigl(16-\frac{320}{21}\bigr)|\nabla R|_p^2
=\frac1{42n(n+2)(n+4)}|\nabla R|_p^2
$$
and a term depending only on~$L$.
Since $L$ is determined by~$\psi^D_0$, we conclude that $\psi^D_0$
and $\psi^D_2$ together determine $|\nabla R|_p^2$\,, as claimed.

In the Neumann case, letting $\psi^N_k(r):=a^N_k(B_r(p))$, we again
have $\psi^N_0(r)=\vol(B_r(p))=\psi^D_0(r)$. Proceeding exactly as in the Dirichlet
case, using~(\ref{eq:a2n}) this time,
we see that $\psi^N_0$ and $\psi^N_2$ together determine the sum of
$$
\frac1{32n(n+2)(n+4)}\bigl(16-\frac{32}3\bigr)|\nabla R|^2_p
=\frac1{6n(n+2)(n+4)}|\nabla R|_p^2
$$
and a term depending only on~$L$. Hence they determine $|\nabla R|_p^2$\,,
as claimed.
\hfill\qed


\begin{thebibliography}{99}

\bibitem{BTV} J. Berndt, F. Tricerri, L. Vanhecke, \emph{Generalized Heisenberg groups and Damek-Ricci harmonic spaces},
\textrm{Lecture Notes in Mathematics 1598}, Springer-Verlag, Berlin/Heidelberg/New York, 1995.

\bibitem{Be} A.L. Besse, \emph{Manifolds all of whose geodesics are closed},
\textrm{Ergebnisse der Mathematik und ihrer Grenzgebiete 93}, Springer-Verlag, Berlin/New York, 1978.

\bibitem{BG} T. Branson, P.B. Gilkey, \emph{The asymptotics of the Laplacian on a manifold with boundary},
\textrm{Comm. Partial Differential Equations} \textbf{15} (1990), no.~2, 245--272.

\bibitem{CV} B.-Y. Chen, L. Vanhecke, \emph{Differential geometry of geodesic spheres},
\textrm{J. Reine Angew. Math.} \textbf{325} (1981), 28--67.

\bibitem{CR} E.T. Copson, H.S. Ruse, \emph{Harmonic Riemannian spaces},
\textrm{Proc. Roy. Soc. Edinburgh} \textbf{60} (1940), 117--133.

\bibitem{DR} E. Damek, F. Ricci, \emph{A class of nonsymmetric harmonic Riemannian spaces},
\textrm{Bull. Amer. Math. Soc. (N.S.)} \textbf{27} (1992), no.~1, 139--142.

\bibitem{DK} D. DeTurck, J. Kazdan, \emph{Some regularity theorems in Riemannian geometry},
\textrm{Ann. scient. \'Ec. Norm. Sup. (4)} \textbf{14} (1981), 249--260.

\bibitem{Fu} H. F\"urstenau, \emph{\"Uber Isospektralit\"at von topologischen B\"allen},
Diploma thesis, Universit\"at Bonn, 2006.

\bibitem{Gi} P. Gilkey, \emph{Invariance Theory, the Heat Equation, and the Atiyah-Singer
Index Theorem}, \textrm{Mathematics Lecture Series~11}, Publish or Perish, Wilmington, Del.,
1984.

\bibitem{Go00} C. Gordon, \emph{Survey of isospectral manifolds}, \textrm{Handbook of differential
geometry}, Vol.~I, 747--778, North-Holland, Amsterdam, 2000.

\bibitem{Go01} C. Gordon, \emph{Isospectral deformations of metrics
on spheres}, \textrm{Invent. Math.} \textbf{145} (2001), 317--331.

\bibitem{GV} A. Gray, L. Vanhecke, \emph{Riemannian geometry as determined by the volumes of
small geodesic balls}, \textrm{Acta Math.} \textbf{142} (1979), no.~3-4, 157--198.

\bibitem{Li44} A. Lichnerowicz, \emph{Sur les espaces Riemanniens compl\`etement
harmoniques}, \textrm{Bull. Soc. Math. France} \textbf{72} (1944), 146--168.

\bibitem{Li58} A. Lichnerowicz, \emph{G\'eom\'etrie des groupes de transformations},
Dunod, Paris, 1958.

\bibitem{NV} L. Nicolodi, L. Vanhecke, \emph{The geometry of $k$-harmonic manifolds},
\textrm{Adv. Geom.} \textbf{6} (2006), no.~1, 53--70.

\bibitem{RWW} H.S. Ruse, A.G. Walker, T.J. Willmore, \emph{Harmonic Spaces},
\textrm{Consiglio Nazionale delle Ricerche, Monografie Matematice 8}, Rome, 1961.

\bibitem{Sa} T. Sakai, \emph{On eigen-values of Laplacian and curvature of Riemannian manifolds},
\textrm{T\^ohoku Math. J. (2)} \textbf{23} (1971), 589--603.

\bibitem{Sz90} Z.I. Szab\'o, \emph{The Lichnerowicz conjecture on harmonic manifolds},
\textrm{J. Differential Geom.} \textbf{31} (1990), no.~1, 1--28.

\bibitem{Sz99} Z.I. Szab\'o, \emph{Locally non-isometric yet super isospectral spaces},
\textrm{Geom. Funct. Anal.} \textbf{9} (1999), no.~1, 185--214.

\bibitem{Sz01} Z.I. Szab\'o, \emph{Isospectral pairs of metrics on balls, spheres, and other
manifolds with different local geometries},
\textrm{Ann. of Math. (2)} \textbf{154} (2001), no.~2, 437--475.

\bibitem{Sz05} Z.I. Szab\'o, \emph{A cornucopia of isospectral pairs of metrics on spheres with different local geometries},  \textrm{Ann. of Math. (2)} \textbf{161}  (2005),  no.~1, 343--395.

\bibitem{Wa} Y. Watanabe, \emph{On the characteristic function of harmonic K\"ahlerian spaces},
\textrm{T\^ohoku Math. J. (2)} \textbf{27} (1975), 13--24.

\end{thebibliography}
\end{document}